\documentclass[10pt,a4paper,oneside]{amsart}

\usepackage{csquotes}
\usepackage{cleveref}
\usepackage{amssymb}

\newcommand{\Z}{\mathbb{Z}}
\newcommand{\nov}{\widehat{\mathbb{Z} G}^\phi}
\newcommand{\minusnov}{\widehat{\mathbb{Z} G}^{-\phi}}

\def\iff{if and only if }

\newtheorem{thm}{Theorem}
\crefname{thm}{Theorem}{Theorems}

\title{Stallings's Fibring Theorem and $\mathrm{PD}^3$-pairs
}
\author{Martin R. Bridson}
\email{bridson@maths.ox.ac.uk}
\author{Dawid Kielak}
\email{kielak@maths.ox.ac.uk}

\address{Mathematical Institute,
	Andrew Wiles Building,
	Observatory Quarter,
	University of Oxford,
	Oxford,
	OX2 6GG,
	United Kingdom}

    \author{Monika Kudlinska}
\email{m.kudlinska@dpmms.cam.ac.uk}
\address{Department of Pure Mathematics and Mathematical Statistics, Centre for Mathematical Sciences, Wilberforce Road, Cambridge,
CB3 0WB, UK}

\begin{document}

\maketitle

\begin{abstract}
	We give a relatively self-contained proof that if a group $G$ fibres algebraically and is part of a $\mathrm{PD}^3$-pair, then $G$ is the fundamental group of a fibred compact aspherical 3-manifold.
    This yields a  homological proof of a classical theorem of Stallings: if $G = \pi_1(M^3)$ is the fundamental group of a compact irreducible 3-manifold $M^3$ and $\phi \colon G \to \Z$ is a surjective homomorphism with finitely generated kernel, then $\phi$ is induced by a topological fibration of $M^3$ over the circle.
\end{abstract}

\bigskip

In this note we present a homological proof of the following result, which was proved by  Stallings \cite{Stallings1962} in the case when $M^3$ is a compact irreducible manifold and the kernel $K$ is not isomorphic to $\Z/2\Z$.

\begin{thm}\label{main}
Let $M^3$ be a compact 3-manifold with aspherical (possibly empty) boundary. Suppose that there exists a short exact sequence
\begin{equation*}
	\label{eq:1} 1 \to K \to \pi_1(M^3) \xrightarrow{\phi} \Z \to 1
\end{equation*}
with $K$ finitely generated. Then $M^3$ is the total space of a fibre bundle
\[ \Sigma \to M^3 \xrightarrow{\Phi} \mathbb{S}^1\]
with $\Phi$ inducing $\phi$.
\end{thm}

Let $G$ be a group and $\mathcal{H}$ a family of (not necessarily distinct) subgroups of $G$. Following Bieri--Eckmann \cite{BieriEckmann1978}, we say that $(G, \mathcal{H})$ is an $n$-dimensional \emph{Poincar\'{e} duality pair}, abbreviated to $\mathrm{PD}^n$-pair, if there exists a $G$-module $C$ which is isomorphic to $\Z$ as an abelian group, called the \emph{dualising module}, and natural isomorphisms
\begin{equation}\label{PDIsoms} \begin{split}H^k(G; M) \simeq H_{n-k}(G, \mathcal{H}; C \otimes M) \\
H^k(G, \mathcal{H}; M) \simeq H_{n-k}(G; C \otimes M)
\end{split}
\end{equation}
for all $G$-modules $M$ and all $k \in \Z$, where we tensor over $\Z$, and the action of $G$ on $C \otimes M$ is the diagonal action. We refer the reader to \cite{BieriEckmann1978} for the definition of relative (co)homology of groups. When $\mathcal{H} = \emptyset$, the above reduces to the definition of a $\mathrm{PD}^n$ group.

The combined work of Eckmann, M\"{u}ller and Linnell in \cite{EckmannMuller1980, EckmannLinnell1983} shows that every $\mathrm{PD}^2$ group is the fundamental group of a closed aspherical 2-manifold. (There are at least two alternate, more modern proofs of this fact, by Bowditch \cite{Bowditch2004} and by P. Kropholler and the second author \cite{KielakKropholler2021}.) From this, one can deduce that any $\mathrm{PD}^2$-pair $(G, \mathcal{H})$ is \emph{geometric} (see the discussion at the start of Section~11 in \cite{BieriEckmann1978}): that is, $G$ is the fundamental group of a compact 2-manifold $\Sigma$ and $\mathcal{H}$ is a (possibly empty) collection of infinite cyclic subgroups of $G$ which correspond to the boundary components of $\Sigma$.

Famously, for $n \geq 3$ it is conjectured that every finitely presented $\mathrm{PD}^n$ group is the fundamental group of an aspherical $n$-manifold. \cref{PD3_group} provides some evidence for this conjecture, and is part of a more general attempt to find homological alternatives to topological arguments. Another interesting example of this approach is a recent article of Reeves, Scott and Swarup~\cite{Reevesetal2020}. We refer the reader
to Hillman's book \cite{Hillman2020} for an  encyclopaedic treatment of this topic.

\begin{thm}\label{PD3_group}
Let $(G, \mathcal{H})$ be a $\mathrm{PD}^3$-pair. If $G$ admits a homomorphism onto $\Z$ with finitely generated kernel then $G$ is the fundamental group of an aspherical 3-manifold $M$, and $\mathcal H$ is the family of the fundamental groups of the boundary components of $M$. Moreover, the manifold $M$ fibres over the circle, and each of the boundary components is either a torus or a Klein bottle.
\end{thm}

When $\mathcal H =  \emptyset$ and the kernel is of type $\mathrm{FP}_2$, this result was proved by Thomas \cite[Theorem 5]{Thomas1984}, and Hillman later showed  \cite[Theorem~1.19]{Hillman2002} that it  suffices to know that the kernel is finitely generated.

The case with boundary, that is, $\mathcal H \neq \emptyset$, is more tricky. Hillman \cite[Theorem 2]{Hillman1987} recovers most of \cref{PD3_group} when the kernel is again of type $\mathrm{FP}_2$. He relies on the following result of Bieri \cite[Corollary 8.6]{Bieri1981}: if $G$ is a group of cohomological dimension $2$, then any normal infinite-index subgroup $K \leqslant  G$ of type $\mathrm{FP}_2$ must be free.
The analogue of Bieri's result is false if one assumes only that $K$ is finitely generated, even in the case $G/K \simeq \Z$, with the Rips construction \cite{Rips1982} providing an ample source of counterexamples.

To circumvent this finiteness issue, one can use the theory of Novikov homology. This was done by Hillman~\cite[Theorem 9.9]{Hillman2020} via the work of Kochloukova~\cite{Kochloukova2006}. We also use Novikov rings in our proof, but in a more direct fashion.

    The \emph{Novikov ring} $\nov$ associated to $\phi$ is the ring obtained by taking all functions $G \to  \Z$ whose support intersected with $\phi^{-1}([-\infty,\kappa))$ is finite for every $\kappa \in \mathbb{R}$. Since $\nov$ contains $\Z G$, multiplication turns it into a $\Z G$-module. The only fact about Novikov rings that we are going to use is the generalised Sikorav's theorem:

\vspace{-0.3cm}

 \begin{equation}\label{SikoravsTheorem}
     \text{    \begin{minipage}[c][2.15 \baselineskip][c]{0.6 \textwidth}
     The group $K = \ker \phi$ is of type $\mathrm{FP}_m$ \iff $H_i(G; \nov) = 0 = H_i(G; \minusnov)$ for every $i \leqslant m$.
    \end{minipage}}
\end{equation}

The case $m=1$ was proved by Sikorav in his thesis \cite{Sikorav1987}; the more general version above is a combination of theorems of Schweitzer \cite[Theorm A.1]{Bieri2007} and Bieri--Renz \cite[Theorem B]{BieriRenz1988}. For a yet more general statement, see the work of S.~Fisher \cite{Fisher2021}.

\begin{proof}[Proof of \cref{PD3_group}]
    Let $(G, \mathcal{H})$ be as in the statement, with $\mathcal{H} = \{H_i\}_{i\in I}$. Let $K$ be the kernel of the epimorphism $\phi \colon G \to \Z$ and suppose that $K$ is finitely generated.
    We start by showing that $K$ is of type $\mathrm{FP}_2$. When $\mathcal H = \emptyset$, this is a special case of Hillman's result \cite[Theorem~1.19]{Hillman2002}; another argument can be extracted from the proof of \cite[Theorem 5.32]{Kielak2020a}. When $\mathcal H \neq \emptyset$, the possibility of some element of $\mathcal H$ lying entirely in $K$ causes a certain amount of annoyance, and stands in the way of using Hillman's proof. The second argument does however generalise, in a rather non-direct way, as follows.

Since we are proving that $K$ is of type $\mathrm{FP}_2$, it is enough to prove the statement for a finite-index subgroup of $K$. Hence we will pass to a subgroup of $G$ of index at most  $2$ that, together with a suitable collection of peripheral subgroups consisting of  subgroups of index $2$ of groups in $\mathcal H$ or two copies of such groups, forms an \emph{orientable} $\mathrm{PD}^3$-pair. We do this in order to make the dualising module  trivial, and hence to avoid worries about tensoring with it. This subgroup of index at most $2$ inherits an epimorphism to $\Z$ with finitely generated kernel, and hence we will lose no generality by assuming that $(G,\mathcal H)$ is an orientable $\mathrm{PD}^3$-pair for this part of the argument.

We now consider a part of the long exact sequence for the group pair $(G,\mathcal H)$:
\[
H^0(\mathcal H; \nov) \to H^1(G, \mathcal H; \nov) \to H^1(G; \nov) \to H^1(\mathcal H; \nov) \to H^2(G, \mathcal H; \nov).
\]
Since $K$ is finitely generated, that is, of type $\mathrm{FP}_1$, we know that $H_i(G; \minusnov)=0$ for $i \in \{0,1\}$. We claim that the cohomology groups $H^i(G; \nov)$ also vanish when $i \in \{0,1\}$ . Indeed, after picking a resolution  of $\Z$ by free finitely generated $\Z G$-modules, this allows us to construct a partial chain contraction over $\minusnov$ witnessing the vanishing of homology up to degree one. Dualising, this gives us a partial chain contraction over $\nov$ that witnesses the vanishing of $H^i(G; \nov)$ for $i \in \{0,1\}$. We plug this information back into the long exact sequence and obtain
\[
H^0(\mathcal H; \nov) \to H^1(G, \mathcal H; \nov) \to 0 \to H^1(\mathcal H; \nov) \to H^2(G, \mathcal H; \nov).
\]
The last term is isomorphic to $H_1(G; \nov) = 0$ by Poincar\'e duality, forcing
\[0 = H^1(\mathcal H; \nov) = \bigoplus_{H \in \mathcal H} H^1(H; \nov).\]
Now, the groups $H \in \mathcal H$ are themselves orientable $\mathrm{PD}^2$ groups by \cite[Theorem 4.2(iii)]{BieriEckmann1978}, and hence are fundamental groups of closed orientable aspherical surfaces. Poincar\'e duality gives us
\[
0 = H^1(H; \nov) \simeq H_1(H; \nov).
\]
If $\phi$ does not vanish on $H$, then $H_0(H;\nov)=0$ as well, and therefore the generalised Sikorav theorem tells us that $K \cap H$ is finitely generated. This implies that $H \simeq \Z^2$, since no other closed orientable aspherical surface has fundamental group that algebraically fibres.

Now let us focus on the initial part of the long exact sequence above. Poincar\'e duality tells us that
\[
H^0(\mathcal H; \nov) = \bigoplus_{H \in \mathcal H} H^0(H; \nov) \simeq   \bigoplus_{H \in \mathcal H} H_2(H; \nov).
\]
If $\phi$ does not vanish on $H$, then we have just shown that $\ker \phi\vert_H \simeq \Z$, which is of type $\mathrm{FP}_2$, and
hence $H_2(H; \nov)=0$. If $H \leqslant K$, then $\nov$ is a flat $\Z H$-module, and so $H_2(H; \nov) = H_2(H; \Z H) \otimes_{\Z H} \nov = 0 \otimes_{\Z H} \nov = 0$. We conclude that
\[H^0(\mathcal H; \nov) = 0,\]
and plugging this back into the long exact sequence yields $H^1(G,\mathcal H; \nov) = 0$. But by Poincar\'e duality, this abelian group is isomorphic to $H_2(G;\nov)$. An analogous argument yields that $H_2(G;\minusnov)=0$, and the generalised Sikorav theorem \eqref{SikoravsTheorem} tells us that $K = \ker \phi$ is of type $\mathrm{FP}_2$, as claimed.

We now return to the general case, that is, we no longer assume $(G,\mathcal H)$ to be orientable; we retain however the newly acquired information that $K$ is of type $\mathrm{FP}_2$.

\smallskip

    For each $i \in I$, let $X_i$ be a system of double coset representatives of $K \backslash G/H_i$. Fix $t \in G$ to be a lift of a generator of $\Z$. Note that for every $H_i \in \mathcal{H}$, the cardinality of $X_i$ is given by the index $n_i$ of $\phi(H_i)$ in $\Z$, and we can take $X_i = \{1, t, \ldots, t^{n_i - 1}\}$ if  $n_i < \infty$, and $X_i = \{ t^i  \mid i \in \Z\}$ otherwise.
    Define the system of subgroups
    \[\mathcal{H}_K = \bigcup_{i \in I} \{K \cap x^{-1} H_i x \mid  x \in X_i\}.\]

    We claim that $(K, \mathcal{H}_K)$ is a $\mathrm{PD}^2$-pair. When $\mathcal{H} = \emptyset$, this follows from \cite[Theorem~1.19]{Hillman2002}; when $\mathcal{H} \neq \emptyset$, it follows from \cite[Theorem~9.9]{Hillman2020}. We will give a direct proof of the claim following the outline of \cite[Theorem 9.11]{Bieri1981}.

    If $\mathcal{H} \neq \emptyset$, then $\mathrm{cd}(G) = 2$ and hence $\mathrm{cd}(K) \leq 2$. If $\mathcal{H} = \emptyset$, then $\mathrm{cd}(K) \leq 2$ by \cite{Strebel1977}. In either case, by the argument above, $K$ is of type $\mathrm{FP}_2$ and thus it is of type $\mathrm{FP}$.

    When $\mathcal{H} = \emptyset$, set $\Delta_{K /\mathcal{H}_K}$ to be the augmentation ideal in $\Z K$.
    When $\mathcal{H} \neq \emptyset$, let $\Delta_{K /\mathcal{H}_K}$ denote the kernel of the augmentation map
     \[ \bigoplus_{i \in I,\, x\in X_i} \mathbb{Z}[K /K \cap x^{-1} H_i x] \to \mathbb{Z}\]
     defined by linearly extending the map which sends each coset $g (K\cap x^{-1}H_i x )$ to a fixed generator of $\Z$. In either case, write $\widetilde{
    \Delta}_{K / \mathcal{H}_K} = C \otimes \Delta_{K / \mathcal{H}_K}$, where $C$ is the dualising module of the $\mathrm{PD}^3$-pair $(G, \mathcal{H})$ and the action of $K \leq G$ on $C$ is the restricted action; the tensor product is taken over $\Z$, and the action of $K$ on it is diagonal.

 Consider the Lyndon--Hochschild--Serre spectral sequence
    \[E_2^{pq} :  H^p(G/K; H^q(K, \Z G)) \Rightarrow H^{p+q}(G; \Z G).\]
    Since $G/K \simeq \Z$, the only values of $p$ yielding groups that are potentially non-zero are $0$ and $1$, and hence the spectral sequence collapses on page $2$. The only terms of the filtration that require further investigation are of the form $H^0( \Z ; H^i(K;\Z G) )$  and $H^1( \Z ; H^{i-1}(K;\Z G) )$.
    Since $K$ is of type $\mathrm{FP}$, we have $H^i(K;\Z G) \simeq H^i(K ;\Z K) \otimes_{\Z K} \Z G$ as $\Z G$-modules.
    In order to compute the cohomology groups above, we need to understand the action of $G/K \simeq \Z$ on $H^i(K ;\Z K) \otimes_{\Z K} \Z G$. To do this, it helps to view $H^i(K ;\Z K) \otimes_{\Z K} \Z G$ as $H^i(K ;\Z K) \otimes_{\Z} \Z (K \backslash G)$. The action of $G/K$ is then diagonal, with the action on $\Z (K \backslash G)$ being multiplication.
    We can now easily compute the $G/K$ invariants and coinvariants of $H^i(K;\Z G)$. Indeed, the invariants form the zero module and the coinvariants are a copy of $H^i(K;\Z K)$, with a $\Z G$-module structure that restricts to the usual $\Z K$-module structure. Using the fact that $\Z$ is an orientable $\mathrm{PD}^1$-group, we conclude from the spectral sequence that
    \[
     H^i(G;\Z G) \simeq  H^{i-1}(K;\Z K)
    \]
    as $\Z G$-modules.

       When $\mathcal H = \emptyset$, the above immediately tells us that $K$ is a $\mathrm{PD}^2$-group, as claimed.
    When $\mathcal H \neq \emptyset$, the pair $(G,\mathcal H)$ is a $\mathrm{PD}^3$-pair, and so \cite[Proposition 6.1]{BieriEckmann1978} tells us that $G$ is a $2$-dimensional duality group with dualising module $H^2(G;\Z G)$ isomorphic to the kernel of the augmentation map
    \[ \bigoplus_{i \in I} \mathbb{Z}[G /H_i] \to \mathbb{Z}\]
    tensored over $\Z$ with $C$.
    This tells us that $K$ is a $1$-dimensional duality group, and a trivial check verifies that the dualising module $H^2(G;\Z G)$ restricted to being a $\Z K$-module is precisely $\widetilde{\Delta}_{K / \mathcal{H}_K}$. This allows us to conclude that $(K, \mathcal H_K)$ is a $\mathrm{PD}^2$-pair using \cite[Proposition 6.1]{BieriEckmann1978} again.

\smallskip

Since $\mathrm{PD}^2$-pairs are geometric, $K$ is the fundamental group of a compact aspherical surface $\Sigma$ with boundary components in one-to-one correspondence with the elements in $\mathcal{H}_K$, and there exists $\alpha \in \mathrm{Out}(\pi_1(\Sigma))$ such that
\[G \simeq \pi_1(\Sigma) \rtimes_{\alpha} \mathbb{Z}.\]

It remains to show that the outer automorphism $\alpha$ of $\pi_1(\Sigma)$ is induced by a homeomorphism of $\Sigma$. When $\Sigma$ is a closed surface, this is a classical result, usually attributed to Dehn, that was proved by Nielsen in \cite{Nielsen1927} in the orientable case and by Mangler  \cite{Mangler1939} in  the non-orientable case.

Suppose now that $\Sigma$ has non-empty boundary. It is clear that the conjugation action of $t$ permutes the $K$-conjugacy classes of the subgroups in $\mathcal{H}_K$, and thus the outer automorphism $\alpha$ permutes the conjugacy classes of the elements of $\pi_1(\Sigma)$ corresponding to the boundary components of $\Sigma$. Now by Theorems~5.7.1 and 5.7.2 in \cite{ZieschangVogtColdewey1980}, $\alpha$ is induced by a homeomorphism $f \colon \Sigma \to \Sigma$. Hence $G$ is isomorphic to the fundamental group of the mapping torus of $f$. The boundary components of the mapping torus are tori or Klein bottles in one-to-one correspondence with the set $I$, and this correspondence identifies their fundamental groups with the groups in $\mathcal H$.
This proves the theorem.
\end{proof}

We say a 3-manifold $M^3$ is \emph{sporadic} if it is an $\mathbb{S}^2$- or $\mathbb{RP}^2$-bundle over the circle.

\begin{proof}[Proof of \cref{main}]
Let $M^3$ be a compact 3-manifold with aspherical (possibly empty) boundary, and suppose that there exists a surjective homomorphism $\phi \colon \pi_1(M^3) \to \Z$ with finitely generated kernel $K$. In particular, this implies that if $\pi_1(M^3) = A \ast B$ or $\pi_1(M^3) = (A \ast B) \rtimes \Z /2 \Z$, then at least one of $A$ or $B$ is the trivial group.

If $M^3$ is sporadic then there are exactly two epimorphisms $\pi_1(M^3) \to \Z$ and the result obviously holds. Hence, from now on we assume that $M^3$ is not sporadic. We claim that it suffices to consider 3-manifolds $M^3$ with incompressible boundary that are \emph{$P^2$-irreducible}, meaning irreducible and not containing an embedded two-sided projective plane.

Indeed, if $M^3$ is non-sporadic and reducible then $M^3$ splits as a connected sum $M^3 \simeq M_1 \# M_2$, where neither $M_1$ nor $M_2$ is the 3-sphere (see \cite[Lemma~3.13]{Hempel1976}). Then, $\pi_1(M) \simeq \pi_1(M_1) \ast \pi_1(M_2)$ and by Perelman's solution of the Poincar\'{e} conjecture, it follows that $\pi_1(M_1)$ and $\pi_1(M_2)$ are non-trivial.

If $M^3$ contains a compressible boundary component, then a standard argument using Papakyriakopoulos' Loop Theorem (see e.g. \cite[Lemma~1.5]{AschenbrennerFriedlWilton2015}) shows that again $\pi_1(M^3)$ admits a non-trivial free product decomposition.

Finally, if a non-sporadic $M^3$ is irreducible and contains a two-sided embedded projective plane, then by \cite{Heil1973} its fundamental group is of the form $(A \ast B) \rtimes \mathbb{Z}/ 2 \mathbb{Z}$ with $A$ and $B$ non-trivial groups. Hence, we may assume that $M^3$ is $P^2$-irreducible and thus by the Sphere Theorem \cite{Papakyriakopoulos1957} and the Projective Plane Theorem \cite{Epstein1961}, combined with the fact that $\pi_1(M^3)$ is infinite, we conclude that $M^3$ is aspherical.

Let $\mathcal{H} = \{H_i\}_{i \in I}$ be the family of subgroups of $\pi_1(M^3)$ corresponding to the boundary components of $M^3$. Then $(\pi_1(M^3), \mathcal{H})$ is a $\mathrm{PD}^3$-pair. Hence by the proof of \cref{PD3_group}, $K \simeq \pi_1(\Sigma)$ for some compact surface $\Sigma$, and there exists a homeomorphism $f \colon \Sigma \to \Sigma$ such that $\pi_1(M^3)$ is isomorphic to the fundamental group of the mapping torus $M_f$ of $f$,
\begin{equation}\label{eq:2} \pi_1(M^3) \simeq \pi_1(M_f) = \pi_1(\Sigma) \rtimes_{f_{*}} \Z. \end{equation}

Now $M_f$ and $M^3$ are $P^2$-irreducible, $M_f$ contains an incompressible surface, and the isomorphism \eqref{eq:2} preserves the peripheral structure of the fundamental groups of the 3-manifolds $M_f$ and $M^3$. Thus by \cite{Heil1969}, there exists a homeomorphism $M^3 \to M_f$ which induces the isomorphism in \eqref{eq:2}. We compose this with the fibration $M_f \to \mathbb{S}^1$ induced by the mapping torus structure, to obtain a topological fibration
\[\Sigma' \to M^3 \to \mathbb{S}^1\]
with $\Sigma'$ homeomorphic to the surface $\Sigma$. This proves \cref{main}.
\end{proof}

\subsection*{Acknowledgements}
The authors thank Jonathan Hillman, Peter Scott, and Gadde Swarup for helpful correspondence.

This work has received funding from the European Research Council (ERC) under the European Union's Horizon 2020 research and innovation programme (Grant agreement No. 850930).

\bibliographystyle{alpha}
\bibliography{refs.bib}

\end{document}